\documentclass{icm2010}

\title[$K$- and $L$-theory of group rings]{$K$- and $L$-theory of group rings}
\author[W. L\"uck]{Wolfgang L\"uck
\thanks{The work was financially supported by the Leibniz-Preis of the author. 
The author wishes to thank several members and guests of the topology group in M\"unster for helpful comments.}
}
\contact[lueck@math.uni-muenster.de \newline URL: http://www.math.uni-muenster.de/u/lueck]
{Mathematisches Institut der 
Westf\"alische Wilhelms-Universit\"at\\
Einsteinstr. 62, 48149 M\"unster, Germany}

\usepackage{pdfsync}
\usepackage{calc}
\usepackage{enumerate,amssymb}
\usepackage{color}
\usepackage[arrow,curve,matrix,tips,2cell]{xy}
 \SelectTips{eu}{10} \UseTips
 \UseAllTwocells
\usepackage{tikz}
\usepackage{pdfcolmk}

\DeclareMathAlphabet{\matheurm}{U}{eur}{m}{n}
\newcommand{\EGF}[2]{E_{#2}(#1)}               % klassifizierender
\newcommand{\eub}[1]{\underline{E}#1}
\newcommand{\edub}[1]{\underline{\underline{E}}#1}

\newcommand{\Or}{\matheurm{Or}}

\DeclareMathOperator{\CAT}{CAT}

\DeclareMathOperator{\class}{class}

\DeclareMathOperator{\colim}{colim}
\DeclareMathOperator{\con}{con}

\DeclareMathOperator{\HS}{HS}

\DeclareMathOperator{\Out}{Out}
\DeclareMathOperator{\pt}{pt}

\DeclareMathOperator{\sign}{sign}

\DeclareMathOperator{\topo}{top}

\DeclareMathOperator{\tr}{tr}
\DeclareMathOperator{\Wh}{Wh}

\newcommand{\Fin}{{\mathcal{F}\text{in}}}
\newcommand{\FJ}{{\mathcal{F}\!\mathcal{J}}}

\newcommand{\VCyc}{{\mathcal{V}\mathcal{C}\text{yc}}}
\newcommand{\Tr}{\mathcal{T}\!\mathcal{R}}
\newcommand{\All}{\mathcal{A}\mathcal{L}\mathcal{L}}

 \newcommand{\IC}{\mathbb{C}}

 \newcommand{\IQ}{\mathbb{Q}}
 \newcommand{\IR}{\mathbb{R}}

 \newcommand{\IZ}{\mathbb{Z}}

 \newcommand{\calf}{\mathcal{F}}

 \newcommand{\calh}{\mathcal{H}}

 \newcommand{\call}{\mathcal{L}}
 \newcommand{\calm}{\mathcal{M}}
 \newcommand{\caln}{\mathcal{N}}

 \newcommand{\cals}{\mathcal{S}}
 
 \newcommand{\calu}{\mathcal{U}}

 \newcommand{\bfE}{{\mathbf E}}

 \newcommand{\bfK}{{\mathbf K}}
 \newcommand{\bfL}{{\mathbf L}}

\newcounter{commentcounter}

\theoremstyle{plain}
\newtheorem{theorem}{Theorem}[section]

\newtheorem{conjecture}[theorem]{Conjecture}

\newtheorem*{theorem*}{Theorem}

\theoremstyle{definition}

\newtheorem{remark}[theorem]{Remark}

\theoremstyle{remark}

\renewcommand{\theenumi}{(\roman{enumi})}

\makeatletter\let\c@equation=\c@theorem\makeatother
\renewcommand{\theequation}{\thetheorem}

\newcommand{\version}[1]                %marks the date of last editing and compilation
{\begin{center} last edited on #1\\
last compiled on \today\\
name of texfile: \jobname
\end{center}
}

\begin{document}

\begin{abstract} This article will explore  the
$K$- and $L$-theory of group rings and their applications to algebra, geometry and topology.
The Farrell-Jones Conjecture characterizes $K$- and $L$-theory groups. It
has many implications, including the Borel and Novikov
Conjectures for topological rigidity.
Its current status, and many of its consequences are surveyed.
\end{abstract}

% Your AMS 200 Classification should come here %

% Any keywords %

% Do not remove next line %
\maketitle

% Now the text of your article starts

%%%%%%%%%%%%%%%%%%%%%%%%%%%%%%%%%%%%%%%%%%%%%%%%%%%%%%%%%%%%%%%%%%%%%%%%%%%%%%%%%
%%%%%%%%%%%%%%%%%%%%%%%%%%%%%%%%%% Introduction %%%%%%%%%%%%%%%%%%%%%%%%%%%%%%%%%
%%%%%%%%%%%%%%%%%%%%%%%%%%%%%%%%%%%%%%%%%%%%%%%%%%%%%%%%%%%%%%%%%%%%%%%%%%%%%%%%%

\typeout{----------------------   Section 0: Introduction ---------------------}

\setcounter{section}{-1}
\section{Introduction}

The algebraic $K$- and $L$-theory of group rings --- $K_n(RG)$ and $L_n(RG)$ for
a ring $R$ and a group $G$ --- are highly significant, but are very hard to
compute when $G$ is infinite. The main ingredient for their analysis is the
Farrell-Jones Conjecture. It identifies them with certain equivariant homology
theories evaluated on the classifying space for the family of virtually cyclic
subgroups of $G$. Roughly speaking, the Farrell-Jones Conjecture predicts that
one can compute the values of these $K$- and $L$-groups for $RG$ if one
understands all of the values for $RH$, where $H$ runs through the virtually
cyclic subgroups of $G$.

Why is the Farrell-Jones Conjecture so important? One reason is that it plays an
important role in the classification and geometry of manifolds. A second reason
is that it implies a variety of well-known conjectures, such as the ones due to
Bass, Borel, Kaplansky and Novikov. (These conjectures are explained in
Section~\ref{sec:Some_well-known_conjectures}.) There are many groups for which
these conjectures were previously unknown but are now consequences of the proof
that they satisfy the Farrell-Jones Conjecture. A third reason is that most of
the explicit computations of $K$- and $L$-theory of group rings for infinite
groups are based on the Farrell-Jones Conjecture, since it identifies them with
equivariant homology groups which are more accessible via standard tools from
algebraic topology and geometry (see Section~\ref{sec:Computational_aspects}).

The rather complicated general formulation of the Farrell-Jones Conjecture is
given in
Section~\ref{sec:The_general_formulation_of_the_Farrell-Jones_Conjecture}. The
much easier, but already very interesting, special case of a torsionfree group
is discussed in
Section~\ref{sec:The_Farrell-Jones_Conjecture_for_torsionfree_groups}. In this
situation the $K$- and $L$-groups are identified with certain homology theories
applied to the classifying space $BG$.

The recent proofs of the Farrell-Jones Conjecture for hyperbolic groups and
$\CAT(0)$-groups are deep and technically very involved. Nonetheless, we give a
glimpse of the key ideas in Section~\ref{sec:Methods_of-Proof}. In each of these
proofs there is decisive input coming from the geometry of the groups that is
reminiscent of non-positive curvature. In order to exploit these geometric
properties one needs to employ controlled topology and construct flow spaces
that mimic the geodesic flow on a Riemannian manifold.

The class of groups for which the Farrell-Jones Conjecture is known is further
extended by the fact that it has certain inheritance properties. For instance,
subgroups of direct products of finitely many hyperbolic groups and directed
colimits of hyperbolic groups belong to this class. Hence, there are many
examples of exotic groups, such as groups with expanders, that satisfy the
Farrell-Jones Conjecture because they are constructed as such colimits. There
are of course groups for which the Farrell-Jones Conjecture has not been proved,
like solvable groups, but there is no example or property of a group known that
threatens to produce a counterexample. Nevertheless, there may well be
counterexamples and the challenge is to develop new tools to find and construct
them.

The status of the Farrell-Jones Conjecture is given in
Section~\ref{sec:The_status_of_the_Farrell-Jones_Conjecture}, and open problems
are discussed in Section~\ref{sec:Open_Problems}.

%%%%%%%%%%%%%%%%%%%%%%%%%%%%%%%%%%%%%%%%%%%%%%%%%%%%%%%%%%%%%%%%%%%%%%%%%%%%%%%%%%%%%%%%%
%%%%%%%%%%%%%%%%%%%%%%%%%%%%%%%%% Section 1 %%%%%%%%%%%%%%%%%%%%%%%%%%%%%%%%%%%%%%%%%%%%%
%%%%%%%%%%%%%%%%%%%%%%%%%%%%%%%%%%%%%%%%%%%%%%%%%%%%%%%%%%%%%%%%%%%%%%%%%%%%%%%%%%%%%%%%%

\typeout{---------------------------------------  Section 1:  --------------------------}

\section{Some well-known conjectures}
\label{sec:Some_well-known_conjectures}

In this section we briefly recall some well-known conjectures.  They address
topics from different areas, including topology, algebra and geometric group
theory.  They have one --- at first sight not at all obvious --- common
feature. Namely, their solution is related to questions about the $K$- and
$L$-theory of group rings.

%%%%%%%%%%%%%%%%%%%%%%%%%%%%%%%%%%%%%%%%%%%%%%%%%%%%%%%%%%%%%%%%%%%%%%%%%%%%%%%%%%%%%%%%%

\subsection{Borel Conjecture}

A closed manifold $M$ is said to be \emph{topologically rigid} if every homotopy
equivalence from a closed manifold to $M$ is homotopic to a homeomorphism.  In
particular, if $M$ is topologically rigid, then every manifold homotopy
equivalent to $M$ is homeomorphic to $M$.  For example, the spheres $S^n$ are
topologically rigid, as predicted by the \emph{Poincar\'e Conjecture}. A
connected manifold is called \emph{aspherical} if its homotopy groups in degree
$\ge 2$ are trivial. A sphere $S^n$ for $n \ge 2$ has trivial fundamental group, 
but its higher homotopy groups are very complicated. Aspherical manifolds,
on the other hand, have complicated fundamental groups and trivial higher
homotopy groups. Examples of closed aspherical manifolds are closed Riemannian
manifolds with non-positive sectional curvature, and double quotients
$G\backslash L /K$ for a connected Lie group $L$ with $K \subseteq L$ a maximal
compact subgroup and $G \subseteq L$ a torsionfree cocompact discrete subgroup. More
information about aspherical manifolds can be found, for instance,
in~\cite{Lueck(2009asph)}.

\begin{conjecture}[Borel Conjecture]
 \label{con:Borel}
 Closed aspherical manifolds are topologically rigid.
\end{conjecture}

In particular the Borel Conjecture predicts that two closed aspherical manifolds
are homeomorphic if and only if their fundamental groups are isomorphic. Hence
the Borel Conjecture may be viewed as the topological version of \emph{Mostow
 rigidity}. One version of Mostow rigidity says that two hyperbolic closed
manifolds of dimension $\ge 3$ are isometrically diffeomorphic if and only if
their fundamental groups are isomorphic.

It is not true that any homotopy equivalence of aspherical closed smooth
manifolds is homotopic to a diffeomorphism.  The $n$-dimensional torus for $n
\ge 5$ yields a counterexample (see~\cite[15A]{Wall(1999)}). Counterexamples
with sectional curvature pinched arbitrarily close to $-1$ are given 
in~\cite[Theorem 1.1]{Farrell-Jones(1989b)}.

For more information about topologically rigid manifolds which are not
necessarily aspherical, the reader is referred to~\cite{Kreck-Lueck(2009nonasph)}.

%%%%%%%%%%%%%%%%%%%%%%%%%%%%%%%%%%%%%%%%%%%%%%%%%%%%%%%%%%%%%%%%%%%%%%%%%%%%%%%%%%%%%%%%%

\subsection{Fundamental groups of closed manifolds}
\label{subsec:Fundamental_groups_of_closed_manifolds}

The Borel Conjecture is a uniqueness result. There is also an existence part.
The problem is to determine when a given group $G$ is the fundamental group of a
closed aspherical manifold. Let us collect some obvious conditions that a group
$G$ must satisfy so that $G = \pi_1(M)$ for a closed aspherical manifold $M$. It
must be finitely presented, since the fundamental group of any closed manifold
is finitely presented.  Since the cellular $\IZ G$-chain complex of the
universal covering of $M$ yields a finite free $\IZ G$-resolution of the trivial
$\IZ G$-module $\IZ$, the group $G$ must be of type FP, i.e., the trivial 
$\IZ G$-module $\IZ$ possesses a finite projective $\IZ G$-resolution.  Since
$\widetilde{M}$ is a model for the classifying $G$-space $EG$, Poincar\'e
duality implies $H^i(G;\IZ G) \cong H_{\dim(M) - i}(\widetilde{M};\IZ)$, where
$H^i(G;\IZ G)$ is the cohomology of $G$ with coefficients in the $\IZ G$-module
$\IZ G$ and $H_i(\widetilde{M};\IZ)$ is the homology of $\widetilde{M}$ with
integer coefficients. Since $\widetilde{M}$ is contractible, $H^i(G;\IZ G) = 0$
for $i \neq \dim(M)$ and $H^{\dim(M)}(G;\IZ G) \cong \IZ$.  Thus, a group $G$ is
called a \emph{Poincar\'e duality group of dimension $n$} if $G$ is finitely
presented, is of type FP, $H^i(G;\IZ G) = 0$ for $i \neq n$, and 
$H^n(G;\IZ G) \cong \IZ$.

\begin{conjecture}[Poincar\'e duality groups]
 \label{con:Poincare_duality_groups}
 A group $G$ is the fundamental group of a closed aspherical manifold of
 dimension $n$ if and only if $G$ is a Poincar\'e duality group of dimension
 $n$.
\end{conjecture}

For more information about Poincar\'e duality 
groups, see~\cite{Davis(2000Poin),Johnson+Wall(1972),Wall(1967)}.

%%%%%%%%%%%%%%%%%%%%%%%%%%%%%%%%%%%%%%%%%%%%%%%%%%%%%%%%%%%%%%%%%%%%%%%%%%%%%%%%%%%%%%%%%

\subsection{Novikov Conjecture}
\label{subsec:Novikov_Conjecture}

Let $G$ be a group and $u \colon M \to BG$ be a map from a closed oriented
smooth manifold $M$ to $BG$. Let $\call(M) \in \prod_{k \ge 0} H^k(M;\IQ)$ be
the \emph{$L$-class of $M$}, which is a certain polynomial in the Pontrjagin
classes. Therefore it depends, a priori, on the tangent bundle and hence on the
differentiable structure of $M$.  For $x \in \prod_{k \ge 0} H^k(BG;\IQ)$, define
the \emph{higher signature of $M$ associated to $x$ and $u$} to be the rational number
\begin{eqnarray*}
 \sign_x(M,u) & := & \langle \call(M) \cup u^*x,[M]\rangle.
\end{eqnarray*}
We say that
$\sign_x$ for $x \in \prod_{n \ge 0} H^n(BG;\IQ)$ is \emph{homotopy invariant} if, for two closed
oriented smooth manifolds $M$ and $N$ with reference maps $u\colon M \to BG$ and
$v \colon N \to BG$, we have
\[
\sign_x(M,u) = \sign_x(N,v)
\]
whenever there is an orientation preserving homotopy equivalence $f \colon M \to N$ 
such that $v \circ f$ and $u$ are homotopic.

\begin{conjecture}[Novikov Conjecture]%
 \label{con:Novikov_Conjecture}
 Let $G$ be a group.  Then $\sign_x$ is homotopy invariant for all $x \in
 \prod_{k \ge 0} H^k(BG;\IQ)$.
\end{conjecture}

The Hirzebruch signature formula says that for $x=1$ the signature
$\sign_1(M,c)$ coincides with the ordinary signature $\sign(M)$ of $M$ if
$\dim(M) = 4n$, and is zero if $\dim(M)$ is not divisible by four.  Obviously
$\sign(M)$ depends only on the oriented homotopy type of $M$ and hence
the Novikov Conjecture~\ref{con:Novikov_Conjecture} is true for $x = 1$. 

A consequence of the Novikov Conjecture~\ref{con:Novikov_Conjecture} is that for 
a homotopy equivalence $f \colon M \to N$ of orientable closed manifolds, we get
$f_*\call(M) = \call(N)$ provided $M$ and $N$ are aspherical. This is surprising since
it is not true in general. Often the $L$-classes are used to distinguish the
homeomorphism or diffeomorphism types of homotopy equivalent closed
manifolds. However, if one believes in the Borel Conjecture~\ref{con:Borel},
then the map $f$ above is homotopic to a homeomorphism and a celebrated result
of Novikov~\cite{Novikov(1965b)} on the topological invariance of
rational Pontrjagin classes says that $f_*\call(M) = \call(N)$ holds for any
homeomorphism of closed manifolds.

For more information about the Novikov Conjecture, see, for instance, 
\cite{Ferry-Ranicki-Rosenberg(1995c), Kreck-Lueck(2005)}.

%%%%%%%%%%%%%%%%%%%%%%%%%%%%%%%%%%%%%%%%%%%%%%%%%%%%%%%%%%%%%%%%%%%%%%%%%%%%%%%%%%%%%%%%%

\subsection{Kaplansky Conjecture}
\label{subsec:Kaplansky_Conjecture}

Let $F$ be a field of characteristic zero. Consider a group $G$. Let $g \in G$
be an element of finite order $|g|$.  Set $N_g = \sum_{i=1}^{|g|} g^i$. Then 
$N_g \cdot N_g = |g| \cdot N_g$. Hence $x = N_g/|g|$ is an idempotent, i.e., $x^2 =x$.  
There are no other constructions known to produce idempotents different from $0$  in $FG$. If
$G$ is torsionfree, this construction yields only the obvious idempotent
$1$. This motivates:

\begin{conjecture}[Kaplansky Conjecture] \label{con:Kaplansky}
 Let $F$ be a field of characteristic zero and let $G$ be a torsionfree group.
 Then the group ring $FG$ contains no idempotents except $0$ and $1$.
\end{conjecture}

%%%%%%%%%%%%%%%%%%%%%%%%%%%%%%%%%%%%%%%%%%%%%%%%%%%%%%%%%%%%%%%%%%%%%%%%%%%%%%%%%%%%%%%%%

\subsection{Hyperbolic groups with spheres as boundary}
\label{subsec:Hyperbolic_groups_with_spheres_as_boundary}

Let $G$ be a hyperbolic group. One can assign to $G$ its boundary $\partial G$.
For information about the boundaries of hyperbolic
groups, the reader is referred to~\cite{Bridson-Haefliger(1999), Kapovich+Benakli(2002),Lueck(2008sgg)}.
Let $M$ be an $n$-dimensional closed connected Riemannian manifold with negative
sectional curvature.  Then its fundamental group $\pi_1(M)$ is a hyperbolic
group.  The exponential map at a point $x \in M$ yields a diffeomorphism $\exp
\colon T_x\IR^n \to M$, which sends $0$ to $x$, and a linear ray emanating from
$0$ in $T_x\IR^n \cong \IR^n$ is mapped to a geodesic ray in $M$ emanating from
$x$. Hence, it is not surprising that the boundary of $\pi_1(M)$ is $S^{\dim(M)
 -1}$.  This motivates (see Gromov~\cite[page~192]{Gromov(1993)}):

\begin{conjecture}[Hyperbolic groups with spheres as boundary]
 \label{con:Hyperbolic_groups_with_spheres_as_boundary}
 Let $G$ be a hyperbolic group whose boundary $\partial G$ is homeomorphic to
 $S^{n-1}$.  Then $G$ is the fundamental group of an aspherical closed manifold
 of dimension $n$.
\end{conjecture}

This conjecture has been proved for $n \ge 6$ by 
Bartels-L\"uck-Weinberger~\cite{Bartels-Lueck-Weinberger(2009)}
using the proof of the Farrell-Jones Conjecture for hyperbolic 
groups (see~\cite{Bartels-Lueck(2009borelhyp)}) and the topology
of homology ANR-manifolds
(see, for example,~\cite{Bryant-Ferry-Mio-Weinberger(1996),Quinn(1987_resolution)}).

%%%%%%%%%%%%%%%%%%%%%%%%%%%%%%%%%%%%%%%%%%%%%%%%%%%%%%%%%%%%%%%%%%%%%%%%%%%%%%%%%%%%%%%%%

\subsection{Vanishing of the reduced projective class group}
\label{subsec:Vanishing_of_the_reduced_projective_class_group}

Let $R$ be an (associative) ring (with unit).  Define its \emph{projective class
 group} $K_0(R)$ to be the abelian group whose generators are isomorphism
classes $[P]$ of finitely generated projective $R$-modules $P$, and whose
relations are $[P_0] + [P_2] = [P_1]$ for any exact sequence $0 \to P_0 \to P_1
\to P_2 \to 0$ of finitely generated projective $R$-modules. Define the
\emph{reduced projective class group} $\widetilde{K}_0(R)$ to be the quotient of
$K_0(R)$ by the abelian subgroup $\{[R^m] - [R^n] \mid n,m \in \IZ, m,n \ge 0\}$,
which is the same as the abelian subgroup generated by the class $[R]$.

Let $P$ be a finitely generated projective $R$-module. Then its class $[P] \in
\widetilde{K}_0(R)$ is trivial if and only if $P$ is \emph{stably free}, i.e.,
$P \oplus R^r \cong R^s$ for appropriate integers $r,s \ge 0$.  So the reduced
projective class group $\widetilde{K}_0(R)$ measures the deviation of a finitely
generated projective $R$-module from being stably free. Notice that stably free
does not, in general, imply free.

A ring $R$ is called \emph{regular} if it is Noetherian and every $R$-module has
a finite-dimensional projective resolution.  Any principal ideal domain, such as
$\IZ$ or a field, is regular.

\begin{conjecture}[Vanishing of the reduced projective class group]
 \label{con:Vanishing_of_the_reduced_projective_class_group}
 Let $R$ be a regular ring and let $G$ be a torsionfree group.  Then the change
 of rings homomorphism
 \[K_0(R) \to K_0(RG)\] is an isomorphism.

 In particular $\widetilde{K}_0(RG)$ vanishes for every principal ideal domain
 $R$ and every torsionfree group $G$.
\end{conjecture}

The vanishing of $\widetilde{K}_0(RG)$ contains valuable information about the
finitely generated projective $RG$-modules over $RG$. In the case $R = \IZ$, it
also has the following important geometric interpretation.

 Let $X$ be a connected $CW$-complex. It is called \emph{finite} if it consists of
 finitely many cells, or, equivalently, if $X$ is compact.  It is called
 \emph{finitely dominated} if there is a finite $CW$-complex $Y$, together with
 maps $i\colon X \to Y$ and $r\colon Y \to X$, such that $r \circ i$ is
 homotopic to the identity on $X$. The fundamental group of a finitely
 dominated $CW$-complex is always finitely presented. While studying existence
 problems for spaces with prescribed properties (like group
 actions, for example), it is occasionally relatively easy to construct a
 finitely dominated $CW$-complex within a given homotopy type, whereas it is
 not at all clear whether one can also find a homotopy equivalent \emph{finite}
 $CW$-complex.  \emph{Wall's finiteness obstruction}, a certain obstruction
 element $\widetilde{o}(X) \in \widetilde{K}_0(\IZ \pi_1(X))$, decides this
 question.

 The vanishing of $\widetilde{K}_0 ( \IZ G )$, as predicted in
 Conjecture~\ref{con:Vanishing_of_the_reduced_projective_class_group} for
 torsionfree groups, has the following interpretation: For a finitely presented
 group $G$, the vanishing of $\widetilde{K}_0(\IZ G)$ is equivalent to the
 statement that any connected finitely dominated $CW$-complex $X$ with $G \cong
 \pi_1(X)$ is homotopy equivalent to a finite $CW$-complex.

 For more information about the finiteness obstruction, 
see~\cite{Ferry-Ranicki(2001),Lueck(1987b),Mislin(1995),Wall(1965a)}.

%%%%%%%%%%%%%%%%%%%%%%%%%%%%%%%%%%%%%%%%%%%%%%%%%%%%%%%%%%%%%%%%%%%%%%%%%%%%%%%%%%%%%%%%%

\subsection{Vanishing of the Whitehead group}
\label{subsec:Vanishing_of_the_Whitehead_group}

The \emph{first algebraic $K$-group} $K_1(R)$ of a ring $R$ is defined to be the
abelian group whose generators $[f]$ are conjugacy classes of automorphisms
$f\colon P \to P$ of finitely generated projective $R$-modules $P$ 
and has the following relations. For each exact sequence 
$0 \to (P_0,f_0) \to (P_1,f_1) \to (P_2,f_2) \to 0$ 
of automorphisms of finitely generated projective
$R$-modules, there is the relation $[f_0] - [f_1] + [f_2] = 0$; and for every two
automorphisms $f,g \colon P \to P$ of the same finitely generated projective
$R$-module, there is the relation $[f \circ g] = [f] + [g]$. Equivalently,
$K_1(R)$ is the abelianization of the general linear group 
$GL(R) = \colim_{n  \to \infty} GL_n(R)$.

An invertible matrix $A$ over $R$ represents the trivial element in $K_1(R)$ if
it can be transformed by elementary row and column operations and by
stabilization, $A \to A \oplus 1$ or the inverse, to the empty matrix.

Let $G$ be a group, and let $\{\pm g\mid g \in G\}$ be the subgroup of $K_1(\IZ G)$
given by the classes of $(1,1)$-matrices of the shape $(\pm g)$ for $g \in
G$. The \emph{Whitehead group $\Wh(G)$ of $G$} is the quotient 
$K_1(\IZ G)/\{\pm g\mid g \in G\}$.

\begin{conjecture}[Vanishing of the Whitehead group]
 \label{con:Vanishing_of_the_Whitehead_group}
The Whitehead group of a torsionfree group vanishes.
\end{conjecture}

This conjecture has the following geometric interpretation.

An \emph{$n$-dimensional cobordism} $(W;M_0,M_1)$ consists of a compact oriented
$n$-di\-men\-sio\-nal smooth manifold $W$ together with a disjoint decomposition
$\partial W = M_0 \coprod M_1$ of the boundary $\partial W$ of $W$.  It is
called an \emph{$h$-cobordism} if the inclusions $M_i \to W$ for $i =0,1$ are
homotopy equivalences. An $h$-cobordism $(W;M_0,M_1)$ is trivial if it is
diffeomorphic relative $M_0$ to the trivial $h$-cobordism 
$(M_0 \times [0,1], M_0 \times \{0\}, M_0 \times \{1\})$.  
One can assign to an $h$-cobordism its
\emph{Whitehead torsion} $\tau(W,M_0)$ in $\Wh(\pi_1(M_0))$.

\begin{theorem}[s-Cobordism Theorem] \label{the:s-cobordism_theorem}
\index{Theorem!s-Cobordism Theorem}
Let $M_0$ be a closed connected oriented smooth manifold of dimension $n \ge 5$
with fundamental group $\pi = \pi_1(M_0)$. Then:
\begin{enumerate}

\item \label{the:s-cobordism_theorem:triviality}
An $h$-cobordism $(W;M_0,M_1)$ is trivial if and only if its Whitehead torsion
$\tau(W,M_0) \in \Wh(\pi)$ vanishes;

\item \label{the:s-cobordism_theorem:realization}
For any $x \in \Wh(\pi)$ there is an $h$-cobordism
$(W;M_0,M_1)$ with $\tau(W,M_0) = x \in \Wh(\pi)$.

\end{enumerate}
\end{theorem}

The $s$-Cobordism Theorem~\ref{the:s-cobordism_theorem} is due to Barden, Mazur,
Stallings.  Its topological version was proved by Kirby and Siebenmann
\cite[Essay II]{Kirby-Siebenmann(1977)}.  More information about the
$s$-Cobordism Theorem can be found, for instance, in \cite{Kervaire(1965)},
\cite[Chapter 1]{Lueck(2002c)}, \cite{Milnor(1965b)}. The Poincar\'e Conjecture of dimension $\ge 5$
is a consequence of the s-Cobordism Theorem~\ref{the:s-cobordism_theorem}. The
s-Cobordism Theorem~\ref{the:s-cobordism_theorem} is an important ingredient in
the surgery theory due to Browder, Novikov, Sullivan and Wall, which is the main
tool for the classification of manifolds. 

The s-Cobordism Theorem tells us that the vanishing of the Whitehead group, as
predicted in Conjecture~\ref{con:Vanishing_of_the_Whitehead_group}, has the
following geometric interpretation: For a finitely presented group $G$ the
vanishing of the Whitehead group $\Wh(G)$ is equivalent to the statement that
every h-cobordism $W$ of dimension $ \ge 6$ with fundamental group
$\pi_1(W)\cong G$ is trivial.

%%%%%%%%%%%%%%%%%%%%%%%%%%%%%%%%%%%%%%%%%%%%%%%%%%%%%%%%%%%%%%%%%%%%%%%%%%%%%%%%%%%%%%%%%

\subsection{The Bass Conjecture}
\label{subsec:The-Bass_Conjecture}

For a finite group $G$ there is a well-known fact that the homomorphism
from the complexification of the complex representation ring of $G$ to the
$\IC$-algebra of complex-valued class functions on $G$, given by taking
the character of a finite-dimensional complex representation, is an isomorphism. The Bass
Conjecture aims at a generalization of this fact to arbitrary groups.

Let $\con(G)$ be the set of conjugacy classes $(g)$ of elements $g \in
G$. Denote by $\con(G)_f$ the subset of $\con(G)$ consisting of those conjugacy
classes $(g)$ for which each representative $g$ has finite order. Let
$\class_0(G)$ and $\class_0(G)_f$ respectively be the $\IC$-vector spaces with the set
$\con(G)$ and $\con(G)_f$ respectively as  basis.  This is the same as the $\IC$-vector space
of $\IC$-valued functions on $\con(G)$ and $\con(G)_f$ with finite support.
Define the \emph{universal $\IC$-trace} as
\begin{eqnarray*} \tr_{\IC G}^u \colon \IC G \to \class_0(G), \quad 
\sum_{g \in G} \lambda_g \cdot g \mapsto \sum_{g \in G} \lambda_g \cdot (g).
\end{eqnarray*}
It extends to a function $\tr_{\IC G}^u \colon M_n(\IC G) \to \class_0(G)$ on
$(n,n)$-matrices over $\IC G$ by taking the sum of the traces of the diagonal
entries.  Let $P$ be a finitely generated projective $\IC G$-module. Choose a
matrix $A \in M_n(\IC G)$ such that $A^2 = A$ and the image of the $\IC G$-map
$r_A \colon \IC G^n \to \IC G^n$ given by right multiplication with $A$ is 
$\IC G$-isomorphic to $P$. Define the \emph{Hattori-Stallings rank} of $P$ as
\begin{eqnarray*}
 \HS_{\IC G}(P) & := & \tr_{\IC G}^u(A)    \in \class_0(G).
\end{eqnarray*}
The Hattori-Stallings rank depends only on the isomorphism class of the $\IC
G$-module $P$ and induces a homomorphism $\HS_{\IC G} \colon K_0(\IC G) \to \class_0(G)$.

\begin{conjecture}[(Strong) Bass Conjecture for $K_0 ( \IC G )$]
\label{con:Strong_Bass_Conjecture_for_K_0(CG)}
The Hattori-Stalling rank yields an isomorphism
\[
\HS_{\IC G} \colon K_0(\IC G) \otimes_{\IZ} \IC \to \class_0(G)_f.
\]
\end{conjecture}

More information and further references 
about the Bass Conjecture can be found
in~\cite[0.5]{Bartels-Lueck-Reich(2008appl)}, \cite{Bass(1976)},\cite[Subsection 9.5.2]{Lueck(2002)}, 
and~\cite[3.1.3]{Lueck-Reich(2005)}.

%%%%%%%%%%%%%%%%%%%%%%%%%%%%%%%%%%%%%%%%%%%%%%%%%%%%%%%%%%%%%%%%%%%%%%%%%%%%%%%%%%%%%%%%%
%%%%%%%%%%%%%%%%%%%%%%%%%%%%%%%%% Section 2 %%%%%%%%%%%%%%%%%%%%%%%%%%%%%%%%%%%%%%%%%%%%%
%%%%%%%%%%%%%%%%%%%%%%%%%%%%%%%%%%%%%%%%%%%%%%%%%%%%%%%%%%%%%%%%%%%%%%%%%%%%%%%%%%%%%%%%%

\typeout{---------------------------------------  Section 2:  --------------------------}

\section{The Farrell-Jones Conjecture for torsionfree \newline groups}
\label{sec:The_Farrell-Jones_Conjecture_for_torsionfree_groups}

%%%%%%%%%%%%%%%%%%%%%%%%%%%%%%%%%%%%%%%%%%%%%%%%%%%%%%%%%%%%%%%%%%%%%%%%%%%%%%%%%%%%%%%%%

\subsection{The $K$-theoretic Farrell-Jones Conjecture for torsionfree 
groups and regular coefficient rings}
\label{subsec:The_K-theoretic_Farrell-Jones_Conjecture_for_torsionfree_group_and_regular_coefficient_rings}

We have already explained $K_0(R)$ and $K_1(R)$ for a ring $R$. There exist
algebraic $K$-groups $K_n(R)$, for every $n \in \IZ$, defined as the
homotopy groups of the associated $K$-theory spectrum $\bfK(R)$. For the
definition of higher algebraic $K$-theory groups and the (connective) $K$-theory
spectrum see, for instance, \cite{Carlsson(2005), Quillen(1973), Rosenberg(1994), Waldhausen(1985)}.  For
information about negative $K$-groups, we refer the reader 
to~\cite{Bass(1968),Farrell-Jones(1995),Pedersen(1984),Pedersen-Weibel(1985),Ranicki(1992a),Rosenberg(1994)}.

How can one come to a conjecture about the structure of the groups $K_n(RG)$?
Let us consider the special situation, where the coefficient ring $R$ is
regular. Then one gets isomorphisms
\begin{eqnarray*}
 K_n(R[\IZ]) & \cong & K_n(R) \oplus K_{n-1}(R);
 \\
 K_n(R[G \ast H]) \oplus K_n(R) & \cong & 
 K_n(RG) \oplus K_n(RH).
\end{eqnarray*}
Now notice that for any generalized homology theory $\calh$, we obtain
isomorphisms
\begin{eqnarray*}
 \calh_n(B\IZ) & \cong & \calh_n(\{\bullet\}) \oplus \calh_{n-1}(\{\bullet\});
 \\
 \calh_n(B(G \ast H)) \oplus \calh_n(\{\bullet\}) & \cong & \calh_n(BG) \oplus \calh(BH).
\end{eqnarray*}
This and other analogies suggest that $K_n(RG)$ may coincide with $\calh_n(BG)$
for an appropriate generalized homology theory.  If this is the case, we
must have $\calh_n(\{\bullet\}) = K_n(R)$. Hence, a natural guess for
$\calh_n$ is $H_n(-;\bfK(R))$, the homology theory associated to the algebraic
$K$-theory spectrum $\bfK(R)$ of $R$.  These considerations lead to:

\begin{conjecture}[$K$-theoretic Farrell-Jones Conjecture for torsionfree groups
 and regular coefficient rings]
 \label{con:K-theoretic_FJC;torsionfree_regular}
 Let $R$ be a regular ring and let $G$ be a torsionfree group. Then there is an
 isomorphism
 \[H_n(BG;\bfK(R)) \xrightarrow{\cong} K_n(RG).
 \]
\end{conjecture}

 \begin{remark}[The Farrell-Jones Conjecture and the vanishing of middle $K$-groups]
  \label{rem:FJC_and_vanishing_midel_K-groups}
   If $R$ is a regular ring, then $K_q(R) = 0$ for $q \le
   -1$. Hence the Atiyah-Hirzebruch spectral sequence converging to
   $H_n(BG;\bfK(R))$ is a first quadrant spectral sequence. Its $E^2$-term is
   $H_p(BG;K_q(R))$.  The edge homomorphism at $(0,0)$ obviously yields an
   isomorphism $H_0(BG;K_0(R)) \xrightarrow{\cong} H_0(BG;\bfK(R))$.  The Farrell-Jones
   Conjecture~\ref{con:K-theoretic_FJC;torsionfree_regular} predicts, because of
   $H_0(BG;K_0(R)) \cong K_0(R)$, that there
   is an isomorphism $K_0(R) \xrightarrow{\cong} K_0(RG)$. We have not
   specified the isomorphism appearing in the Farrell-Jones
   Conjecture~\ref{con:K-theoretic_FJC;torsionfree_regular} above. However, we
   remark that it is easy to check that this isomorphism $K_0(R) \xrightarrow{\cong}
   K_0(RG)$ must be the change of rings map associated to the inclusion $R
   \to RG$.  Thus, we see that the Farrell-Jones
   Conjecture~\ref{con:K-theoretic_FJC;torsionfree_regular} implies
   Conjecture~\ref{con:Vanishing_of_the_reduced_projective_class_group}.

   The Atiyah-Hirzebruch spectral sequence yields an exact sequence 
  $0 \to K_1(R) \to H_1(BG;\bfK(R)) \to H_1(G,K_0(R)) \to 0$.  In the special case $R =
   \IZ$, this reduces to an exact sequence $0 \to \{\pm 1\} \to H_1(BG;\bfK(R))
   \to G/[G,G] \to 0$.  This implies that  the assembly map 
   sends $H_1(BG;\bfK(R))$ bijectively onto the subgroup 
   $\{\pm g\mid g \in G\}$ of $K_1(\IZ G)$. Hence, the Farrell-Jones
   Conjecture~\ref{con:K-theoretic_FJC;torsionfree_regular} implies
   Conjecture~\ref{con:Vanishing_of_the_Whitehead_group}.
 \end{remark}

  \begin{remark}[The Farrell-Jones Conjecture and the Kaplansky Conjecture]
  \label{rem:FJC_and_Kaplanky}
  The Farrell-Jones Conjecture~\ref{con:K-theoretic_FJC;torsionfree_regular}
  also implies the Kaplansky Conjecture~\ref{con:Kaplansky} 
  (see~\cite[Theorem~0.12]{Bartels-Lueck-Reich(2008appl)}).
\end{remark}

 \begin{remark}[The conditions torsionfree and regular are needed in Conjecture~\ref{con:K-theoretic_FJC;torsionfree_regular}]
   \label{rem:FJC_special_not_true_in_general}
   The version of the Farrell-Jones
   Conjecture~\ref{con:K-theoretic_FJC;torsionfree_regular} cannot be true
   without the assumptions that $R$ is regular and $G$ is torsionfree. The
   Bass-Heller-Swan decomposition yields an isomorphism $K_n(R[\IZ]) \cong
   K_n(R) \oplus K_{n-1}(R) \oplus N\!K_n(R) \oplus N\!K_n(R)$, whereas
   $H_n(B\IZ;\bfK(R)) \cong K_n(R) \oplus K_{n-1}(R)$. If $R$ is regular, then
   $N\!K_n(R)$ is trivial, but there are rings $R$ with non-trivial $N\!K_n(R)$.

   Suppose that $R = \IC$ and $G$ is finite. Then $H_0(BG;\bfK_{\IC}) \cong
   K_0(\IC) \cong \IZ$, whereas $K_0(\IC G)$ is the complex representation ring
   of $G$, which is isomorphic to $\IZ$ if and only if $G$ is trivial.
 \end{remark}

%%%%%%%%%%%%%%%%%%%%%%%%%%%%%%%%%%%%%%%%%%%%%%%%%%%%%%%%%%%%%%%%%%%%%%%%%%%%%%%%%%%%%%%%%

\subsection{The $L$-theoretic Farrell-Jones Conjecture for torsionfree groups}
\label{subsec:The_L-theoretic_Farrell-Jones_Conjecture_for_torsionfree_group_and_regular_coefficient_rings}

There is also an $L$-theoretic version of
Conjecture~\ref{con:K-theoretic_FJC;torsionfree_regular}:

\begin{conjecture}[$L$-theoretic Farrell-Jones Conjecture for torsionfree
 groups]
 \label{con:L-theoretic_FJC_torsionfree}
 Let $R$ be a ring with involution and let $G$ be a torsionfree group. 
 Then there is an
 isomorphism
 \[H_n\bigl(BG;\bfL(R)^{\langle - \infty\rangle}\bigr) \xrightarrow{\cong} L_n^{\langle -
   \infty\rangle}(RG).
 \]
\end{conjecture}

Here $\bfL(R)^{\langle - \infty\rangle}$ is the periodic quadratic $L$-theory
spectrum of the ring with involution $R$ with decoration $\langle -
\infty\rangle$, and $L_n^{\langle - \infty\rangle}(R)$ is the $n$-th quadratic
$L$-group with decoration $\langle - \infty\rangle$, which can be identified
with the $n$-th homotopy group of $\bfL_{RG}^{\langle - \infty\rangle}$.  For more
information about the various types of $L$-groups and decorations and $L$-theory
spectra we refer the reader to~\cite{Cappell-Ranicki-Rosenberg(2000),
 Cappell-Ranicki-Rosenberg(2001), Quinn(1970),Ranicki(1981), Ranicki(1992),
 Ranicki(1992a),Ranicki(2002),Wall(1999)}.  Roughly speaking,
$L$-theory deals with quadratic forms. For even $n$, $L_n(R)$ is related to the
Witt group of quadratic forms and for odd $n$, $L_n(R)$ is related to automorphisms of
quadratic forms.  Moreover, the $L$-groups are four-periodic, i.e., $L_n(R)
\cong L_{n+4}(R)$.

\begin{theorem}[The Farrell-Jones Conjecture implies the Borel Conjecture in
 dimensions $\ge 5$]
 \label{the:The_FJC_implies_the_Borel_Conjecture_in_dimension_ge_5}
 Suppose that a torsionfree group $G$ satisfies
 Conjecture~\ref{con:K-theoretic_FJC;torsionfree_regular} and
 Conjecture~\ref{con:L-theoretic_FJC_torsionfree} for $R = \IZ$. Then the Borel
 Conjecture~\ref{con:Borel} holds for any closed aspherical manifold of
 dimension $\ge 5$ whose fundamental group is isomorphic to~$G$.
\end{theorem}
\begin{proof}[Sketch of proof]
  The \emph{topological structure set} $\cals^{\topo}(M)$ of a closed manifold
  $M$ is defined to be the set of equivalence classes of homotopy equivalences
  $f\colon M' \to M$, with a topological closed manifold as its source and $M$
  as its target, for which $f_0 \colon M_0 \to M$ and $f_1 \colon M_1 \to M$ are
  equivalent if there is a homeomorphism $g\colon M_0 \to M_1$ such that $f_1
  \circ g$ and $f_0$ are homotopic.  The Borel Conjecture~\ref{con:Borel} can be
  reformulated in the language of surgery theory to the statement that
  $\cals^{\topo}(M)$ consists of a single point if $M$ is an aspherical closed
  topological manifold.

 The \emph{surgery sequence} of a closed topological manifold $M$ of dimension $n \ge
 5$ is the exact sequence
 \begin{multline}
   \ldots \to \caln_{n+1}(M\times [0,1],M \times \{0,1\}) \xrightarrow{\sigma}
   L^s_{n+1}(\IZ\pi_1(M)) \xrightarrow{\partial} \cals^{\topo}(M)
   \\
   \xrightarrow{\eta} \caln_n(M) \xrightarrow{\sigma} L_n^s(\IZ\pi_1(M)),
 \label{surgery_sequence}
 \end{multline}
 which extends infinitely to the left.  It is the fundamental tool for the
 classification of topological manifolds.  (There is also a smooth version of
 it.)  The map $\sigma$ appearing in the sequence sends a normal map of degree
 one to its surgery obstruction.  This map can be identified with the version
 of the $L$-theory assembly map, where one works with the $1$-connected cover
 $\bfL^s ( \IZ ) \langle 1 \rangle$ of $\bfL^s( \IZ )$.  The map $H_k(M;\bfL^s
 (\IZ)\langle 1 \rangle ) \to H_k(M;\bfL^s (\IZ))$ is injective for $k=n$ and
 an isomorphism for $k >n$. Because of the $K$-theoretic assumptions (and the
 so-called Rothenberg sequence), we can replace the $s$-decoration with the
 $\langle - \infty \rangle$-decoration.  Therefore the Farrell-Jones
 Conjecture~\ref{con:L-theoretic_FJC_torsionfree} implies that the map
 $\sigma\colon \caln_n(M) \to L_n^s(\IZ\pi_1(M))$ is injective and the map 
 $\caln_{n+1}(M\times   [0,1],M \times \{0,1\}) 
 \xrightarrow{\sigma} L^s_{n+1}(\IZ\pi_1(M))$ 
 is bijective. Thus, by the surgery sequence, 
 $\cals^{\topo}(M)$ is a point and hence the Borel Conjecture~\ref{con:Borel}
 holds for $M$. More details can be found in \cite[pages
 17,18,28]{Ferry-Ranicki-Rosenberg(1995)}, \cite[Chapter 18]{Ranicki(1992)}.
\end{proof}

For more information about surgery theory, see
\cite{Cappell-Ranicki-Rosenberg(2000),
 Cappell-Ranicki-Rosenberg(2001), 
 Kreck(1999),Lueck(2002c),Ranicki(2002),Wall(1999)}.

%%%%%%%%%%%%%%%%%%%%%%%%%%%%%%%%%%%%%%%%%%%%%%%%%%%%%%%%%%%%%%%%%%%%%%%%%%%%%%%%%%%%%%%%%
%%%%%%%%%%%%%%%%%%%%%%%%%%%%%%%%%%%% Section 3 %%%%%%%%%%%%%%%%%%%%%%%%%%%%%%%%%%%%%%%%%%%%%%
%%%%%%%%%%%%%%%%%%%%%%%%%%%%%%%%%%%%%%%%%%%%%%%%%%%%%%%%%%%%%%%%%%%%%%%%%%%%%%%%%%%%%%%%%

\typeout{---------------------------------------  Section 3:  --------------------------}

\section{The general formulation of the Farrell-Jones Conjecture}
\label{sec:The_general_formulation_of_the_Farrell-Jones_Conjecture}

%%%%%%%%%%%%%%%%%%%%%%%%%%%%%%%%%%%%%%%%%%%%%%%%%%%%%%%%%%%%%%%%%%%%%%%%%%%%%%%%%%%%%%%%%

\subsection{Classifying spaces for families}
\label{subsec:Classiyfing_spaces_for_families}

Let $G$ be a group. A \emph{family $\calf$ of subgroups of $G$} is a set of
subgroups which is closed under conjugation with elements of $G$ and under
taking subgroups.  A $G$-CW-complex, all of whose isotropy groups belong to
$\calf$ and whose $H$-fixed point sets are contractible for all $H \in \calf$,
is called a \emph{classifying space for the family $\calf$} and will be denoted
$\EGF{G}{\calf}$.  Such a space is unique up to $G$-homotopy, because it is
characterized by the property that for any $G$-$CW$-complex $X$, all whose
isotropy groups belong to $\calf$, there is precisely one $G$-map from $X$ to
$\EGF{G}{\calf}$ up to $G$-homotopy. These spaces were introduced by tom Dieck
\cite{Dieck(1972)}. A functorial ``bar-type'' construction is given in
\cite[section 7]{Davis-Lueck(1998)}.

The space $\EGF{G}{\Tr}$, for $\Tr$ the family consisting of the trivial subgroup
only, is the same as the space $EG$, which is by definition the total space of the
universal $G$-principal bundle $G \to EG \to BG$, or, equivalently, the
universal covering of $BG$.  A model for $\EGF{G}{\All}$, for the family $\All$
of all subgroups, is given by the space $G/G =\{\bullet\}$ consisting of one
point.

The space $\EGF{G}{\Fin}$, for $\Fin$ the family of finite subgroups, is also
known as \emph{the classifying space for proper $G$-actions}, and is denoted by
$\eub{G}$ in the literature. Recall that a $G$-$CW$-complex $X$ is proper if and only
if all of its isotropy groups are finite (see for instance \cite[Theorem 1.23 on
page 18]{Lueck(1989)}).

There are often nice models for $\eub G$. If $G$ is word hyperbolic in the sense
of Gromov, then the Rips-complex is a finite model~\cite{Meintrup-Schick(2002)}.  
If $G$ is a discrete subgroup of a Lie group $L$
with finitely many path components, then for any maximal compact subgroup $K
\subseteq L$, the space $L/K$ with its left $G$-action is a model for $\eub{G}$. 
More information about $\eub{G}$ can be
found in
\cite{Baum-Connes-Higson(1994),Echterhof-Lueck-Philipps-Walters(2010),
Lueck(2000a),Lueck(2005s),Lueck-Meintrup(2000)}.

Let $\VCyc$ be the family of virtually cyclic subgroups, i.e., subgroups which
are either finite or contain $\IZ$ as subgroup of finite index.  We often
abbreviate $\edub{G} = \EGF{G}{\VCyc}$.

%%%%%%%%%%%%%%%%%%%%%%%%%%%%%%%%%%%%%%%%%%%%%%%%%%%%%%%%%%%%%%%%%%%%%%%%%%%%%%%%%%%%%%%%%

\subsection{$G$-homology theories}
\label{subsec:G-homology?theories}

Fix a group $G$. A \emph{$G$-homology theory $\calh_*^G$} is a collection of
covariant functors $\calh^G_n$ from the category of $G$-$CW$-pairs to the
category of abelian groups indexed by $n \in \IZ$ together with natural
transformations
$$\partial_n^G(X,A)\colon \calh_n^G(X,A) \to
\calh_{n-1}^G(A):= \calh_{n-1}^G(A,\emptyset)$$ for $n \in \IZ$, such that four
axioms hold; namely, $G$-homotopy invariance, long exact sequence of a pair,
excision, and the disjoint union axiom. The obvious formulation of these axioms
is left to the reader or can be found in~\cite{Lueck(2002b)}.  Of course a
$G$-homology theory for the trivial group $G=\{1 \}$ is a homology theory
(satisfying the disjoint union axiom) in the classical non-equivariant sense.

\begin{remark}[$G$-homology theories and spectra over $\Or(G)$]
\label{rem:G-homology_theories_spectra_over_Or(G)}
The \emph{orbit category} $\Or (G)$ has as objects the homogeneous spaces $G/H$
and as morphisms $G$-maps. Given a covariant functor $\bfE$ from $\Or(G)$ to the
category of spectra, there exists a $G$-homology theory $\calh^G_*$ 
such that $\calh^G_n(G/H) = \pi_n(\bfE(G/H))$ holds for all $n \in \IZ$ and
subgroups $H \subseteq G$ (see~\cite{Davis-Lueck(1998)}, \cite[Proposition~6.3
on page~737]{Lueck-Reich(2005)}). For trivial $G$,  this boils down to the
classical fact that a spectrum defines a homology theory.
\end{remark}

%%%%%%%%%%%%%%%%%%%%%%%%%%%%%%%%%%%%%%%%%%%%%%%%%%%%%%%%%%%%%%%%%%%%%%%%%%%%%%%%%%%%%%%%%

\subsection{The Meta-Isomorphism Conjecture}
\label{subsec:The_Meta-Isomorphism_Conjecture}

Now we can formulate the following Meta-Conjecture for a group $G$,
a family of subgroups $\calf$, and a $G$-homology theory $\calh^G_*$.

\begin{conjecture}[Meta-Conjecture] \label{con:Meta-Conjecture}
The so-called assembly map
\[
A_{\calf} \colon \calh^G_n (\EGF{G}{\calf}) \to \calh^G_n( \pt ),
\]
which is the map induced by the projection $\EGF{G}{\calf} \to \pt$,
is an isomorphism for $n \in \IZ$.
\end{conjecture}

Notice that the Meta-Conjecture~\ref{con:Meta-Conjecture} is always true if we choose
$\calf = \All$. So given $G$ and $\calh^G_*$, the point is to choose
$\calf$ as small as possible.

%%%%%%%%%%%%%%%%%%%%%%%%%%%%%%%%%%%%%%%%%%%%%%%%%%%%%%%%%%%%%%%%%%%%%%%%%%%%%%%%%%%%%%%%%

\subsection{The Farrell-Jones Conjecture}
\label{subsec:The_Farrell-Jones-Conjecture}

Let $R$ be a ring. Then one can construct for every group $G$, using
Remark~\ref{rem:G-homology_theories_spectra_over_Or(G)}, $G$-homology theories
$H^G_*(-;\bfK_R)$ and $H^G_*\bigl(-;\bfL_R^{\langle -\infty
 \rangle}\bigr)$ satisfying $H^G_n(G/H;\bfK_R) \cong K_n(RH)$ and
$H^G_n\bigl(G/H;\bfL^{\langle -\infty \rangle}_R\bigr) 
\cong L^{\langle - \infty\rangle}_n(RH)$.  The Meta-Conjecture~\ref{con:Meta-Conjecture} 
for $\calf = \VCyc$ is the Farrell-Jones Conjecture:

\begin{conjecture}[Farrell-Jones Conjecture]
 \label{con:Farrell-Jones-Conjecture}
 The maps induced by the projection $\edub{G} \to G/G$ are, for every $n \in \IZ$,
 isomorphisms
 \begin{align*}
   H_n^G(\edub{G};\bfK_R) & \to H_n^G(G/G;\bfK_R) = K_n(RG);
   \\
   H_n^G\bigl(\edub{G};\bfL^{\langle -\infty \rangle}_R\bigr) & \to
   H_n^G\bigl(G/G;\bfL^{\langle -\infty \rangle}_R\bigr) = L^{\langle -\infty
     \rangle}_n(RG).
 \end{align*}
\end{conjecture}

The version of the Farrell-Jones Conjecture~\ref{con:Farrell-Jones-Conjecture}
is equivalent to the original version due to Farrell-Jones~\cite[1.6 on page
257]{Farrell-Jones(1993a)}. The decoration $\langle - \infty \rangle$ cannot be
replaced by the decorations $h$, $s$ or $p$ in general, since there are
counterexamples for these decorations (see~\cite{Farrell-Jones-Lueck(2002)}).

\begin{remark}[Generalized Induction
 Theorem]\label{rem:Generalized_Induction_Theorem}
 One may interpret the Farrell-Jones Conjecture as a kind of generalized
 induction theorem. A prototype of an induction theorem is Artin's Theorem,
 which essentially says that the complex representation ring of a finite group
 can be computed in terms of the representation rings of the cyclic
 subgroups. In the Farrell-Jones setting one wants to compute $K_n(RG)$ and
 $L_n^{\langle -\infty \rangle}(RG)$ in terms of the values of these functors
 on virtually cyclic subgroups, where one has to take into account all the
 relations coming from inclusions and conjugations, and the values in degree
 $n$ depend on all the values in degree $k \le n$ on virtually cyclic
 subgroups.  
\end{remark}

\begin{remark}[The choice of the family $\VCyc$]
 \label{rem:The_choice_of_the_family_aVcyc}
 One can show that, in general, $\VCyc$ is the smallest family of subgroups for
 which one can hope that the Farrell-Jones Conjecture is true for all $G$ and
 $R$. The family $\Fin$ is definitely too small.  Under certain conditions one
 can use smaller families, for instance, $\Fin$ is sufficient if $R$ is regular
 and contains $\IQ$, and $\Tr$ is sufficient if $R$ is regular and $G$ is
 torsionfree.  This explains that Conjecture~\ref{con:Farrell-Jones-Conjecture}
 reduces to Conjecture~\ref{con:K-theoretic_FJC;torsionfree_regular} and
 Conjecture~\ref{con:L-theoretic_FJC_torsionfree}.  More information about
 reducing the family of subgroups can be found in
 \cite{Bartels-Lueck(2007ind)}, \cite{Davis-Khan-Ranicki(2008)}, 
 \cite{Davis-Quinn-Reich(2010)}, \cite[Lemma~4.2]{Lueck(2005heis)},
 \cite[2.2]{Lueck-Reich(2005)}, \cite{Quinn(2005)}. 
 \end{remark}

 Remarks~\ref{rem:Generalized_Induction_Theorem} 
 and~\ref{rem:The_choice_of_the_family_aVcyc} can be illustrated by the 
 following consequence
 of the Farrell-Jones Conjecture~\ref{con:Farrell-Jones-Conjecture}: Given
 a field $F$ of characteristic zero and a group $G$, the obvious map
 $$
 \bigoplus_{H \subseteq G, |H| < \infty} K_0(FH) \to K_0(FG)
 $$
 coming from the various inclusions $H \subseteq G$ is surjective, and actually
 induces an isomorphism
 $$
 \colim_{H \subseteq G, |H| < \infty} \; K_0(FH) \xrightarrow{\cong} K_0(FG).
 $$

 \begin{remark}[The $K$-theoretic Farrell-Jones Conjecture and the Bass
   Conjecture]
   \label{rem:FJC_and_Bass}
   The $K$-theoretic Farrell-Jones
   Conjecture~\ref{con:Farrell-Jones-Conjecture} implies the Bass
   Conjecture~\ref{con:Farrell-Jones-Conjecture}
   (see~\cite[Theorem~0.9]{Bartels-Lueck-Reich(2008appl)}).
 \end{remark}

 \begin{remark}[Coefficients in additive categories]
   \label{rem:coefficients_in_additive_categories}
   It is sometimes important to consider twisted group rings, where we take a
   $G$-action on $R$ into account, or more generally, crossed product rings $R
   \ast G$. In the $L$-theory case we also want to allow orientation
   characters. All of these generalizations can be uniformly handled if one
   allows coefficients in an additive category.  These more general versions of
   the Farrell-Jones Conjectures are explained for $K$-theory
   in~\cite{Bartels-Reich(2007coeff)} and for $L$-theory
   in~\cite{Bartels-Lueck(2009coeff)}.  These generalizations also encompass the
   so-called fibered versions. One of their main features is that they have much
   better inheritance properties, (e.g., passing to subgroups, direct and free
   products, directed colimits) than the untwisted
   version~\ref{con:Farrell-Jones-Conjecture}.

   For proofs the coefficients are often dummy variables. In the right setup it
   does not matter whether one uses coefficients in a ring $R$ or in an additive
   category.
 \end{remark}

%%%%%%%%%%%%%%%%%%%%%%%%%%%%%%%%%%%%%%%%%%%%%%%%%%%%%%%%%%%%%%%%%%%%%%%%%%%%%%%%%%%%%%%%%

\subsection{The Baum-Connes and the Bost Conjectures}
\label{subsec:The_Baum-Connes-Conjecture_and_the_Bost_Conjecture}

There also exists a $G$-homology theory
$H^G_*\bigl(-;\bfK^{\topo}_{C^*_r}\bigr)$ with the property that
$H^G_n\bigl(G/H;\bfK^{\topo}_{C^*_r}\bigr) = K_n(C^*_r(H))$, where
$K_n(C^*_r(H))$ is the topological $K$-theory of the reduced group
$C^*$-algebra. For a proper $G$-$CW$-complex $X$, the equivariant topological
$K$-theory $K_n^G(X)$ agrees with $H^G_n\bigl(X;\bfK^{\topo}_{C^*_r}\bigr)$. The
Meta-Conjecture~\ref{con:Meta-Conjecture} for $\calf = \Fin$ is:

\begin{conjecture}[Baum-Connes-Conjecture]
 \label{con:Baum-Connes-Conjecture}
 The maps induced by the projection $\eub{G} \to G/G$
 \[
 K_n^G(\eub{G}) = H_n^G\bigl(\eub{G};\bfK_{C^*_r}^{\topo}\bigr) \to
 H_n^G\bigl(G/G;\bfK_{C^*_r}^{\topo}\bigr) = K_n(C^*_r(G)).
 \]
 are isomorphisms for every $n \in \IZ$.
\end{conjecture}

The original version of the Baum-Connes Conjecture is stated in~\cite[Conjecture
3.15 on page 254]{Baum-Connes-Higson(1994)}.  For more information about the
Baum-Connes Conjecture, see, for instance,~\cite{Higson(1998a),Lueck-Reich(2005),Mislin-Valette(2003)}.

\begin{remark}[The relation between the conjectures of Novikov, Farrell-Jones
 and Baum-Connes]
 \label{rem:relations}
 Both the $L$-theoretic Farrell-Jones
 Conjecture~\ref{con:Farrell-Jones-Conjecture} and the Baum-Connes
 Conjecture~\ref{con:Baum-Connes-Conjecture} imply the Novikov Conjecture.
 See~\cite[Section~23]{Kreck-Lueck(2005)}, where the relation between
 the $L$-theoretic Farrell-Jones Conjecture~\ref{con:Farrell-Jones-Conjecture} and the
 Baum-Connes Conjecture~\ref{con:Baum-Connes-Conjecture} is also explained.
\end{remark}

%%%%%%%%%%%%%%%%%%%%%%%%%%%%%%%%%%%%%%%%%%%%%%%%%%%%%%%%%%%%%%%%%%%%%%%%%%%%%%%%%%%%%%%%%
%%%%%%%%%%%%%%%%%%%%%%%%%%%%%%%%%%%% Section 4 %%%%%%%%%%%%%%%%%%%%%%%%%%%%%%%%%%%%%%%%%%
%%%%%%%%%%%%%%%%%%%%%%%%%%%%%%%%%%%%%%%%%%%%%%%%%%%%%%%%%%%%%%%%%%%%%%%%%%%%%%%%%%%%%%%%%

\typeout{---------------------------------------  Section 4:  --------------------------}

\section{The status of the Farrell-Jones Conjecture}
\label{sec:The_status_of_the_Farrell-Jones_Conjecture}

%%%%%%%%%%%%%%%%%%%%%%%%%%%%%%%%%%%%%%%%%%%%%%%%%%%%%%%%%%%%%%%%%%%%%%%%%%%%%%%%%%%%%%%%%

\subsection{The work of Farrell-Jones and the status in 2004}

One of the highlights of the work of Farrell and Jones is their proof of the
Borel Conjecture~\ref{con:Borel} for manifolds of dimension $\geq 5$ which
support a Riemannian metric of non-positive sectional
curvature~\cite{Farrell-Jones(1993c)}.  They were able to extend this result to
cover compact complete affine flat manifolds of dimension $\geq
5$~\cite{Farrell-Jones(1998)}.  This was done by considering complete
non-positively curved manifolds that are not necessarily compact. Further
results by Farrell and Jones about their conjecture for $K$-theory and
pseudo-isotopy can be found in~\cite{Farrell-Jones(1993a)}. For a detailed
report about the status of the Baum-Connes Conjecture and Farrell-Jones
Conjecture in 2004 we refer to~\cite[Chapter~5]{Lueck-Reich(2005)}, where one
can also find further references to relevant papers.

%%%%%%%%%%%%%%%%%%%%%%%%%%%%%%%%%%%%%%%%%%%%%%%%%%%%%%%%%%%%%%%%%%%%%%%%%%%%%%%%%%%%%%%%%

\subsection{Hyperbolic groups and $\CAT(0)$-groups}

In recent years, the class of groups for which the
Farrell-Jones Conjecture, and hence the other conjectures appearing in
Section~\ref{sec:Some_well-known_conjectures}, are true has been extended considerably beyond
fundamental groups of non-positively curved manifolds. In what follows, a
\emph{hyperbolic group} is to be understood in the sense of Gromov. A
$\CAT(0)$-group is a group that admits a proper isometric cocompact action on
some $\CAT(0)$-space of finite topological dimension. 

\begin{theorem}[Hyperbolic groups]\label{the:FJC_for_hyperbolic_groups}
 The Farrell-Jones Conjecture with coefficients in additive categories (see
 Remark~\ref{rem:coefficients_in_additive_categories}) holds for both $K$- and
 $L$-theory for every hyperbolic group.
\end{theorem}
\begin{proof}
 The $K$-theory part is proved in
 Bartels-L\"uck-Reich~\cite{Bartels-Lueck-Reich(2008hyper)}, the $L$-theory
 part in Bartels-L\"uck~\cite{Bartels-Lueck(2009borelhyp)}.
\end{proof}

\begin{theorem}[$\CAT(0)$-groups]\label{the:FJC_for_CAT(0)_groups}\
 \begin{enumerate}

 \item
   \label{the:FJC_for_CAT(0)_groups:L}
   The $L$-theoretic Farrell-Jones Conjecture with coefficients in additive
   categories (see Remark~\ref{rem:coefficients_in_additive_categories}) holds
   for every $\CAT(0)$-group;

 \item
   \label{the:FJC_for_CAT(0)_groups:K}
   The assembly map for the $K$-theoretic Farrell-Jones Conjecture with
   coefficients in additive categories (see
   Remark~\ref{rem:coefficients_in_additive_categories}) is bijective in
   degrees $n \le 0$ and surjective in degree $n = 1$ for every
   $\CAT(0)$-group.
 \end{enumerate}
\end{theorem}
\begin{proof}
 This is proved in Bartels-L\"uck~\cite{Bartels-Lueck(2009borelhyp)}.
\end{proof}

For the proofs that the Farrell-Jones Conjecture implies the conjectures
mentioned in Section~\ref{sec:Some_well-known_conjectures}, it suffices to know
the statements appearing in Theorem~\ref{the:FJC_for_CAT(0)_groups}. For instance
Theorem~\ref{the:FJC_for_CAT(0)_groups} implies the Borel Conjecture for every
closed aspherical manifold of dimension $\ge 5$ whose fundamental group is a
$\CAT(0)$-group.

%%%%%%%%%%%%%%%%%%%%%%%%%%%%%%%%%%%%%%%%%%%%%%%%%%%%%%%%%%%%%%%%%%%%%%%%%%%%%%%%%%%%%%%%%

\subsection{Inheritance properties}
\label{sec:Inheritance_Properties}

We have already mentioned that the version of the Farrell-Jones Conjecture with
coefficients in additive categories (see
Remark~\ref{rem:coefficients_in_additive_categories}) does not only include
twisted group rings and allow one to insert orientation homomorphisms, but it
also has very valuable inheritance properties.

\begin{theorem}[Inheritance properties]
 \label{the:inheritance_properties}
 Let (A) be one of the following assertions for a group $G$:

 \begin{itemize}

 \item The $K$-theoretic Farrell-Jones Conjecture with coefficients in additive
   categories (see Remark~\ref{rem:coefficients_in_additive_categories}) holds
   for $G$;

 \item The $K$-theoretic Farrell-Jones Conjecture with coefficients in additive
   categories (see Remark~\ref{rem:coefficients_in_additive_categories}) holds
   for $G$ up to degree one, i.e., the assembly map is bijective in dimension
   $n \le 0$ and surjective for $n = 1$;

 \item The $L$-theoretic Farrell-Jones Conjecture with coefficients in additive
   categories (see Remark~\ref{rem:coefficients_in_additive_categories}) holds
   for $G$.

 \end{itemize}

 Then the following is true:

 \begin{enumerate}

 \item \label{sec:Inheritance_Properties:subgroups} If $G$ satisfies assertion
   (A), then also every subgroup $H \subseteq G$ satisfies (A);

 \item \label{sec:Inheritance_Properties:products} If $G_1$ and $G_2$ satisfies
   assertion (A), then also the free product $G_1 \ast G_2$ and the direct
   product $G_1 \times G_2$ satisfy assertion (A);

 \item \label{sec:Inheritance_Properties:extensions} Let $\pi \colon G \to Q$
   be a group homomorphism.  If $Q$ satisfies (A) and for every virtually cyclic subgroup $V
   \subseteq Q$, its preimage $\pi^{-1}(V)$ satisfies (A), then $G$ satisfies
   assertion (A);

 \item \label{sec:Inheritance_Properties:colimits} Let $\{ G_i \; | \; i \in I
   \}$ be a directed system of groups (with not necessarily injective structure
   maps). If each $G_i$ satisfies assertion (A), then the colimit $\colim_{i
     \in I} G_i$ satisfies assertion (A).

 \end{enumerate}
\end{theorem}
\begin{proof}
 See~\cite[Lemma~2.3]{Bartels-Lueck(2009borelhyp)}.
\end{proof}

%%%%%%%%%%%%%%%%%%%%%%%%%%%%%%%%%%%%%%%%%%%%%%%%%%%%%%%%%%%%%%%%%%

\subsection*{Examples}

Let $\FJ$ be the class of groups satisfying both the $K$-theoretic and
$L$-theoretic Farrell-Jones Conjecture with additive categories as coefficients
(see Remark~\ref{rem:coefficients_in_additive_categories}).  Let $\FJ_{\le 1}$ be the
class of groups which satisfy the $L$-theoretic Farrell-Jones Conjecture with
additive categories as coefficients and the $K$-theoretic Farrell-Jones
Conjecture with additive categories as coefficients up to degree one.

In view of the results above, these classes contain many groups which
lie in the region \emph{Hic Abundant Leones} in Martin Bridson's
universe of groups (see~\cite{Bridson(2006ICM)}).
Theorem~\ref{the:FJC_for_hyperbolic_groups} and
Theorem~\ref{sec:Inheritance_Properties}~\ref{sec:Inheritance_Properties:colimits}
imply that directed colimits of hyperbolic groups belong to $\FJ$.  This class
of groups contains a number of groups with unusual properties.  Counterexamples
to the Baum-Connes Conjecture with coefficients are groups with
expanders~\cite{Higson-Lafforgue-Skandalis(2002)}.  The only known construction
of such groups is as a directed colimit of hyperbolic groups
(see~\cite{Arzhantseva-Delzant(2008)}).  Thus the Farrell-Jones Conjecture in
$K$- and $L$-theory holds for the only presently known counterexamples to the
Baum-Connes Conjecture with coefficients.  (We remark that the formulation of
the Farrell-Jones Conjecture we are considering allows for twisted group rings,
so this includes the correct analog of the Baum-Connes Conjecture with
coefficients.)  The class of directed colimits of hyperbolic groups contains, for
instance, a torsionfree non-cyclic group all whose proper subgroups are cyclic,
constructed by Ol'shanskii~\cite{Olshanskii(1979)}.  Further examples are lacunary groups
(see~\cite{Olshanskii-Osin-Sapir(2009)}).

Davis and Januszkiewicz used Gromov's hyperbolization technique to construct
exotic aspherical manifolds.  They showed that for every $n \geq 5$
there are closed aspherical $n$-dimensional manifolds such that
their universal covering is a
$\CAT(0)$-space whose fundamental group at infinity is
non-trivial~\cite[Theorem~5b.1]{Davis-Januszkiewicz(1991)}.  In particular,
these universal coverings are not homeomorphic to Euclidean space.  Because these
examples are non-positively curved polyhedron, their fundamental
groups are $\CAT(0)$-groups and belong to $\FJ_{\le 1}$.  There is
a variation of this construction that uses the strict hyperbolization of
Charney-Davis~\cite{Charney-Davis(1995)} and produces closed aspherical
manifolds whose universal cover is not homeomorphic to Euclidean space and whose
fundamental group is hyperbolic.  All of these examples are topologically rigid.

\emph{Limit groups} in the sense of Zela have been a
focus of geometric group theory in recent years.  
Alibegovi\'c-Bestvina~\cite{Alibegovic+Bestvina(2006)} have shown that limit groups are
$\CAT(0)$-groups.

Let $G$ be a (not necessarily cocompact) lattice in $SO(n,1)$, e.g., the
fundamental group of a hyperbolic Riemannian manifold with finite volume.  Then
$G$ acts properly cocompactly and isometrically on a $\CAT(0)$-space
by~\cite[Corollary~11.28 in Chapter~II.11 on page~362]{Bridson-Haefliger(1999)},
and hence belongs to $\FJ_{\le1}$.

%%%%%%%%%%%%%%%%%%%%%%%%%%%%%%%%%%%%%%%%%%%%%%%%%%%%%%%%%%%%%%%%%%%%%%%%%%%%%%%%%%%%%%%%%
%%%%%%%%%%%%%%%%%%%%%%%%%%%%%%%%%%%%% Section 5 %%%%%%%%%%%%%%%%%%%%%%%%%%%%%%%%%%%%%%%%%
%%%%%%%%%%%%%%%%%%%%%%%%%%%%%%%%%%%%%%%%%%%%%%%%%%%%%%%%%%%%%%%%%%%%%%%%%%%%%%%%%%%%%%%%%

\typeout{---------------------------------------  Section 5:  --------------------------}

\section{Computational aspects}
\label{sec:Computational_aspects}

It is very hard to compute $K_n(RG)$ or $L_n^{\langle -\infty \rangle}(RG)$
directly. It is easier to compute the source of the assembly map appearing in
the Farrell-Jones Conjecture~\ref{con:Farrell-Jones-Conjecture}, since one can
apply standard techniques for the computation of equivariant homology theories
and there are often nice models for $\eub{G}$. Rationally, equivariant Chern
characters, as developed in~\cite{Lueck(2002b),Lueck(2002d),Lueck(2005c)} give
rather general answers. We illustrate this with the following result taken
from~\cite[Example~8.11]{Lueck(2002b)}.

\begin{theorem} \label{the:complex_coefficients} Let $G$ be a group for which
 the Farrell-Jones Conjecture~\ref{con:Farrell-Jones-Conjecture} holds for $R =
 \IC$.  Let $T$ be the set of conjugacy classes $(g)$ of elements $g \in G$ of
 finite order. For an element $g \in G$, denote by $C_G\langle g \rangle$ the
 centralizer of $g$. Then we obtain isomorphisms
 \begin{eqnarray*}
   \bigoplus_{p+q=n} \bigoplus_{(g) \in T}
   H_p(C_G\langle g \rangle;\IC) \otimes_{\IZ}  K_q(\IC) & \to &
   \IC \otimes_{\IZ} K_n(\IC G);
   \\
   \bigoplus_{p+q=n} \bigoplus_{(g) \in T}
   H_p(C_G\langle g \rangle;\IC) \otimes_{\IZ}  L_q^{\langle -\infty \rangle}(\IC) & \to &
   \IC \otimes_{\IZ} L_n^{\langle -\infty \rangle}(\IC G),
 \end{eqnarray*}
 where we use the involutions coming from complex
 conjugation in the definition of $L_q^{\langle -\infty \rangle}(\IC)$ and
 $L_n^{\langle -\infty \rangle}(\IC G)$.
\end{theorem}

Integral computations can only be given in special cases. For example, the
semi-direct product $\IZ^r \rtimes \IZ/n$ cannot be handled in general. Not even
its ordinary group homology is known, so it is not a surprise that the $K$- and
$L$-theory of the associated group ring are unknown in general.  Sometimes
explicit answers can be found in the literature, see for
instance~\cite[8.3]{Lueck-Reich(2005)}.  As an illustration we mention the
following result which follows from Theorem~\ref{the:FJC_for_hyperbolic_groups}
using~\cite[Theorem~1.3]{Bartels(2003b)}, and~\cite[Corollary~2.11 and
Example~3.6]{Lueck-Weiermann(2007)}.

\begin{theorem}[Torsionfree hyperbolic groups]
\label{the:torsionfree_hyperbolic_groups}
Let $G$ be a torsionfree hyperbolic group.
Let $\calm$ be a complete system of representatives of 
the conjugacy classes of maximal infinite cyclic subgroups of $G$.
\begin{enumerate}
\item \label{the:torsionfree_hyperbolic_groups:K}
For every $n \in \IZ$, there is an isomorphism
\begin{eqnarray*}
H_n\bigl(BG;\bfK(R)\bigr) \oplus \bigoplus_{V \in \calm}  N\!K_n(R) \oplus  N\!K_n(R) 
& \xrightarrow{\cong} & 
K_n(RG),
\end{eqnarray*}
where $N\!K_n(R)$ the \emph{Bass-Nil-groups of $R$};  

\item \label{the:torsionfree_hyperbolic_groups:L}
For every $n \in \IZ$, there is an isomorphism
\begin{eqnarray*}
H_n\bigl(BG;\bfL^{\langle -\infty \rangle}(R)\bigr)  
& \xrightarrow{\cong} &
L_n^{\langle -\infty \rangle}(RG).
\end{eqnarray*}
\end{enumerate}
\end{theorem}

Computations of $L$-groups of group rings are important in the classification of
manifolds since they appear in the surgery sequence~\eqref{surgery_sequence}.

%%%%%%%%%%%%%%%%%%%%%%%%%%%%%%%%%%%%%%%%%%%%%%%%%%%%%%%%%%%%%%%%%%%%%%%%%%%%%%%%%%%%%%%%%
%%%%%%%%%%%%%%%%%%%%%%%%%%%%%%%%%%%%% Section 6 %%%%%%%%%%%%%%%%%%%%%%%%%%%%%%%%%%%%%%%%%
%%%%%%%%%%%%%%%%%%%%%%%%%%%%%%%%%%%%%%%%%%%%%%%%%%%%%%%%%%%%%%%%%%%%%%%%%%%%%%%%%%%%%%%%%

\typeout{---------------------------------------  Section 6:  --------------------------}

\section{Methods of proof}
\label{sec:Methods_of-Proof}

Here is a brief sketch of the strategy of proof which has led to the results
about hyperbolic groups and $\CAT(0)$-groups mentioned above.  It is influenced
by ideas of Farrell and Jones. However, we have to deal with spaces that are not
manifolds, and hence new ideas and techniques are required. A more detailed
survey about methods of proof can be found
in~\cite[Section~1]{Bartels-Lueck(2009borelhyp)},~\cite[Section~1]{Bartels-Lueck-Reich(2008cover)},%
~\cite[Section~1]{Bartels-Lueck-Reich(2008hyper)},~\cite{Lueck(2009Hangzhou)}
and~\cite[Chapter~7]{Lueck-Reich(2005)}.

\subsubsection*{Assembly and forget control}
We have defined the assembly map appearing in the Farrell-Jones Conjecture as a
map induced by the projection $\edub{G} \to G/G$. A homotopy theoretic
interpretation by homotopy colimits and a description in terms of the universal
property that it is the best approximation from the left by a homology theory is
presented in \cite{Davis-Lueck(1998)}. This interpretation is good for
structural and computational aspects but is not helpful for actual proofs. For
this purpose the interpretation of the assembly map as a \emph{forget control
  map} is the right one. This fundamental idea is due to Quinn.

Roughly speaking, one attaches to a metric space certain categories, to these
categories spectra and then takes their homotopy groups, where everything
depends on a choice of certain control conditions which in some sense measure
sizes of cycles.  If one requires certain control conditions, one obtains the
source of the assembly map.  If one requires no control conditions, one obtains the
target of the assembly map.  The assembly map itself is forgetting the control
condition.

One of the basic features of a homology theory is excision. It often comes from
the fact that a representing cycle can be found with arbitrarily good control.
An example is the technique of subdivision which allows to make the representing
cycles for simplicial homology arbitrarily controlled. That is, the diameter of
any simplex appearing with non-zero coefficient is very small.  One may say that
requiring control conditions amounts to implementing homological properties.

With this interpretation it is clear what the main task in the proof of
surjectivity of the assembly map is: \emph{achieve control}, i.e., manipulate
cycles without changing their homology class so that they become sufficiently
controlled.  There is a general principle that a proof of surjectivity also
gives injectivity, Namely, proving injectivity means that one must construct a
cycle whose boundary is a given cycle, i.e., one has to solve a surjectivity
problem in a relative situation.

\subsubsection*{Contracting maps and open coverings}
Contracting maps on suitable control spaces are very useful for gaining
control. The idea is that the contraction improves the control of a cycle without
changing its homology class if the contracting map is, roughly speaking, homotopic
to the identity.  Of course one has to choose the contracting maps 
and control spaces with care. If a $G$-space $X$ has a fixed
point, the projection to this fixed point is a contracting $G$-equivariant map,
but it turns out that this is just enough to prove the trivial
version of the Meta Conjecture, where the family $\calf$ is not $\VCyc$ as
desired, but rather consists of all subgroups.

Let $\calf$ be a family of subgroups and let $X$ be a metric space with an
isometric $G$-action.  An \emph{$\calf$-covering} $\calu$ is an open covering
$\calu$ such that $gU \in \calu$ holds for $U \in \calu, g \in G$, for every $U
\in \calu$ and $g \in G$ we have $gU \cap U \not= \emptyset \implies gU = U$,
and for every $U \in \calu$ the subgroup $G_U = \{g \in G \mid gU = U\}$ belongs
to $\calf$.  Associated to these data there is a map $f_{\calu} \colon X \to
|\calu|$ from $X$ to the simplicial nerve of $\calu$.  The larger the Lebesgue
number of $\calu$ is, the more contracting the map becomes with respect to the
$L^1$-metric on $|\calu|$, provided we are able to fix a uniform bound on its
covering dimension (see~\cite[Proposition~5.3]{Bartels-Lueck-Reich(2008hyper)}).

Notice that the simplicial nerve carries a $G$-$CW$-complex structure and all
its isotropy groups belong to $\calf$. We see that $\calf$-coverings can yield
contracting maps, as long as the covering dimension of the possible $\calu$ are
uniformly bounded.

An axiomatic description of the properties such an equivariant covering has to
fulfill can be found in~\cite[Section~1]{Bartels-Lueck-Reich(2008hyper)} and
more generally in~\cite[Section~1]{Bartels-Lueck(2009borelhyp)}.  The
equivariant coverings satisfy conditions that are similar to those for finite
asymptotic dimension, but with extra requirements about equivariance.  A key
technical paper for the construction of such equivariant coverings
is~\cite{Bartels-Lueck-Reich(2008cover)}, where the connection to asymptotic
dimension is explained.

\subsubsection*{Enlarging $G$ and transfer}
Let us try to find $\calf$-coverings for $G$ considered as a metric
space with the word metric. If we take $\calu = \{G\}$, we obtain a
$G$-invariant open covering with arbitrarily large Lebesgue number, but
the open set $G$ is an $\calf$-set only if we take $\calf$ to be the
family of all subgroups.  If we take $\calu = \bigl\{\{g\} \mid g \in
G\bigr\}$ and denote by $\Tr$ the family consisting only of the trivial subgroup, 
we obtain a $\Tr$-covering of topological dimension zero,
but the Lebesgue number is not very impressive, it's just $1$. In order to
increase the Lebesgue number, we could take large balls around each
element. Since the covering has to be $G$-invariant, we could start
with $\calu = \bigl\{B_R(g) \mid g \in G\bigr\}$, where $B_R(g)$ is
the open ball of radius $R$ around $g$.  This is a $G$-invariant open
covering with Lebesgue number $R$, but the sets $B_R(g)$ are not
$\calf$-sets in general and the covering dimension grows with $R$.

One of the main ideas is not to cover $G$ itself, but to enlarge $G$ to
$G \times \overline{X}$ for an appropriate compactification
$\overline{X}$ of a certain contractible metric space $X$ that has an
isometric proper cocompact $G$-action.  This allows us to spread out the open
sets and avoid having too many intersections. This strategy has also
been successfully used in measurable group theory, where the role of
the topological space $\overline{X}$ is played by a probability space
with measure preserving $G$-action (see
Gromov~\cite[page~300]{Gromov(2007)}).

The elements under consideration lie in $K$- or $L$-theory spaces
associated to the control space $G$.  Using a transfer they can be lifted
to $G \times \overline{X}$. (This step
corresponds in the proofs of Farrell and Jones to the passage to the
sphere tangent bundle.) We gain control there and then push the
elements down to $G$. Since the space $\overline{X}$ is contractible,
its Euler characteristic is $1$ and hence the composite of the
push-down map with the transfer map is the identity on the $K$-theory
level.  On the $L$-theory level one needs something with signature $1$.
On the algebra level this
corresponds to the assignment of a finitely generated projective
$\IZ$-module $P$ to its \emph{multiplicative hyperbolic form $H_{\otimes}(P)$}.
It is given by replacing $\oplus$ by $\otimes$ in the standard definition of a hyperbolic form, 
i.e., the underlying $\IZ$-module is $P^*
\otimes P$ and the symmetric form is given by the formula $(\alpha,p)
\otimes (\beta,q) \mapsto \alpha(q) \cdot \beta(p)$. Notice that the
signature of $H_{\otimes}(\IZ)$ is $1$ and taking the multiplicative hyperbolic
form yields an isomorphism of rings $K_0(\IZ) \to L^0(\IZ)$.

We can construct $\VCyc$-coverings that are contracting in the
$G$-direction but will actually expand in the
$\overline{X}$-direction.  The latter defect can be compensated for
because the transfer yields elements over $G \times \overline{X}$ with
arbitrarily good control in the $\overline{X}$-direction.

\subsubsection*{Flows} To find such coverings of $G \times
\overline{X}$, it is crucial to construct, for hyperbolic and
$\CAT(0)$-spaces, flow spaces $\operatorname{FS}(X)$ which are the
analog of the geodesic flow on a simply connected
Riemannian manifold with negative or non-positive sectional
curvature. One constructs appropriate
coverings on $\operatorname{FS}(X)$, often called \emph{long and thin coverings}, 
and then pulls them back with a
certain map $G \times \overline{X} \to \operatorname{FS}(X)$. 
The flow is used to improve a
given covering.  The use flow spaces to gain control is one of the fundamental
ideas of Farrell and Jones (see for instance~\cite{Farrell-Jones(1986a)}).

Let us look at a special example to illustrate the use of a flow.  Consider two
points with coordinates $(x_1,y_1)$ and $(x_2,y_2)$ in the upper half plane
model of two-dimensional hyperbolic space. We want to use the geodesic flow to
make their distance smaller in a functorial fashion.  This is achieved by
letting these points flow towards the boundary at infinity along the geodesic
given by the vertical line through these points, i.e., towards infinity in the
$y$-direction. There is a fundamental problem: if $x_1 =x_2$, then the distance of these points is
unchanged. Therefore we make the following prearrangement. Suppose that $y_1 <
y_2$. Then we first let the point $(x_1,y_1)$ flow so that it reaches a position
where $y_1 = y_2$. Inspecting the
hyperbolic metric, one sees that the distance between the two points
$(x_1,\tau)$ and $(x_2,\tau)$ goes to zero if $\tau$ goes to infinity. This is
the basic idea to gain control in the negatively curved case.

Why is the non-positively curved case harder? Again, consider the upper half
plane, but this time equip it with the flat Riemannian metric coming from
Euclidean space. Then the same construction makes sense, but the distance
between two points $(x_1,\tau)$ and $(x_2,\tau)$ is unchanged if we change
$\tau$.  The basic first idea is to choose a focal point far away, say $f :=
\bigl((x_1 + x_2)/2,\tau + 169356991\bigr)$, and then let $(x_1,\tau)$ and
$(x_2,\tau)$ flow along the rays emanating from them and passing through the
focal point $f$. In the beginning the effect is indeed that the distance becomes
smaller, but as soon as we have passed the focal point the distance grows
again. Either one chooses the focal point very far away or uses the idea of
moving the focal point towards infinity while the points flow.  Roughly
speaking, we are suggesting the idea of a \emph{dog and sausage} principle. We
have a dog, and attached to it is a long stick pointing in front of it with a
delicious sausage on the end. The dog will try to reach the sausage, but the
sausage is moving away according to the movement of the dog, so the dog will
never reach the sausage. (The dog will become long and thin this way, but this
is a different effect).  The problem with this idea is obvious, we must describe
this process in a functorial way and carefully check all the estimates to
guarantee the desired effects.

%%%%%%%%%%%%%%%%%%%%%%%%%%%%%%%%%%%%%%%%%%%%%%%%%%%%%%%%%%%%%%%%%%%%%%%%%%%%%%%%%%%%%%%%%
%%%%%%%%%%%%%%%%%%%%%%%%%%%%%%%%% Section 7 %%%%%%%%%%%%%%%%%%%%%%%%%%%%%%%%%%%%%%%%%%%%%
%%%%%%%%%%%%%%%%%%%%%%%%%%%%%%%%%%%%%%%%%%%%%%%%%%%%%%%%%%%%%%%%%%%%%%%%%%%%%%%%%%%%%%%%%

\typeout{---------------------------------------  Section 7:  --------------------------}

\section{Open Problems}
\label{sec:Open_Problems}

\subsection{Virtually poly-cyclic groups, cocompact lattices and $3$-manifold
 groups}
It is conceivable that our methods can be used to show that virtually
poly-cyclic groups belong to $\FJ$ or $\FJ_{\le 1}$. This already implies the
same conclusion for cocompact lattices in almost connected Lie groups following
ideas of Farrell-Jones~\cite{Farrell-Jones(1993a)} and for fundamental groups of
(not necessarily compact) $3$-manifolds (possibly with boundary) following ideas
of Roushon~\cite{Roushon(2008FJJ3)}.

%%%%%%%%%%%%%%%%%%%%%%%%%%%%%%%%%%%%%%%%%%%%%%%%%%%%%%%%%%%%%%%%%%%%%%%%%%%%%%%%%%%%%%%%%

\subsection{Solvable groups}

Show that solvable groups belong to $\FJ$ or $\FJ_{\le 1}$.  In view of the
large class of groups belonging to $\FJ$ or $\FJ_{\le
 1}$, it is very surprising that it is not known whether a semi-direct product
$A \rtimes_{\varphi} \IZ$ for a (not necessarily finitely generated) abelian group $A$
belongs to $\FJ$ or $\FJ_{\le 1}$. The problem is the possibly complicated dynamics of
the automorphism $\varphi$ of $A$.

Such groups are easy to handle in the Baum-Connes setting, where one can use the
long exact Wang sequence for topological $K$-theory associated to a semi-direct
product. Such a sequence does not exists for algebraic $K$-theory, and new
contributions involving Nil-terms occur.

%%%%%%%%%%%%%%%%%%%%%%%%%%%%%%%%%%%%%%%%%%%%%%%%%%%%%%%%%%%%%%%%%%%%%%%%%%%%%%%%%%%%%%%%%

\subsection{Other open cases}
Show that mapping class groups, $\Out(F_n)$ and Thompson's groups belong to
$\FJ$ or $\FJ_{\le 1}$.  The point here is not that this has striking
consequence in and of itself, but rather their proofs will probably give more insight in the
Farrell-Jones Conjecture and will require some new input about the geometry
of these groups which may be interesting in its own right.

A very interesting open case is $SL_n(\IZ)$. The main obstacle is that
$SL_n(\IZ)$ does not act cocompactly isometrically properly on a
$\CAT(0)$-space; the canonical action on $SL_n(\IR)/SO(n)$ is proper and
isometric and of finite covolume but not cocompact.  The Baum-Connes Conjecture
is also open for $SL_n(\IZ)$.

%%%%%%%%%%%%%%%%%%%%%%%%%%%%%%%%%%%%%%%%%%%%%%%%%%%%%%%%%%%%%%%%%%%%%%%%%%%%%%%%%%%%%%%%%

\subsection{Searching for counterexamples}
There is no group known for which the Farrell-Jones Conjecture is false. There
has been some hope that groups with expanders may yield counterexamples, but
this hope has been dampened since colimits of hyperbolic groups satisfy it.  At
the moment one does not know any property of a group which makes it likely to
produce a counterexample. The same holds for the Borel Conjecture. Many of the
known exotic examples of closed aspherical manifolds are known to satisfy the
Borel Conjecture.

In order to find counterexamples one seems to need completely new ideas, maybe
from random groups or logic.

%%%%%%%%%%%%%%%%%%%%%%%%%%%%%%%%%%%%%%%%%%%%%%%%%%%%%%%%%%%%%%%%%%%%%%%%%%%%%%%%%%%%%%%%%

\subsection{Pseudo-isotopy}
Extend our results to pseudo-isotopy spaces. There are already interesting
results for these proved by Farrell-Jones~\cite{Farrell-Jones(1993a)}.

%%%%%%%%%%%%%%%%%%%%%%%%%%%%%%%%%%%%%%%%%%%%%%%%%%%%%%%%%%%%%%%%%%%%%%%%%%%%%%%%%%%%%%%%%

\subsection{Transfer of methods}
The Baum-Connes Conjecture is unknown for all
$\CAT(0)$-groups.  Can one use the techniques of the proof of the
Farrell-Jones Conjecture for $\CAT(0)$-groups to prove the Baum-Connes
Conjecture for them? In particular it is not at all clear how the
transfer methods in the Farrell-Jones setting carry over to the
Baum-Connes case. In the other direction, the Dirac-Dual Dirac method,
which is the main tool for proofs of the Baum-Connes Conjecture, lacks
an analog on the Farrell-Jones side.

%%%%%%%%%%%%%%%%%%%%%%%%%%%%%%%%%%%%%%%%%%%%%%%%%%%%%%%%%%%%%%%%%%%%%%%%%%%%%%%%%%%%%%%%%

\subsection{Classification of (non-aspherical) manifolds}
The Farrell-Jones Conjecture is also very useful when one considers not
necessarily aspherical manifolds. Namely, because of the surgery
sequence~\eqref{surgery_sequence}, it gives an interpretation of the structure set
as a relative homology group.  So it simplifies the classification of
manifolds substantially and opens the door to explicit answers in favorable
interesting cases. Here, a lot of work can and will have to be done.

%%%%%%%%%%%%%%%%%%%%%%%%%%%%%%%%%%%%%%%%%%%%%%%%%%%%%%%%%%%%%%%%%%%%%%%%%%%%%%%%%%%%%%%%%
%%%%%%%%%%%%%%%%%%%%%%%%%%%%%%%%%%%% References %%%%%%%%%%%%%%%%%%%%%%%%%%%%%%%%%%%%%%%%%
%%%%%%%%%%%%%%%%%%%%%%%%%%%%%%%%%%%%%%%%%%%%%%%%%%%%%%%%%%%%%%%%%%%%%%%%%%%%%%%%%%%%%%%%%

\typeout{---------------------------------------  References --------------------------}

%\bibliographystyle{abbrv}
%\bibliography{dbpub,dbpre}
%\version{24.03.2010}

\end{document}